\documentclass[12pt,thmsa]{article}
\usepackage{amsmath, latexsym, amsfonts, amssymb, amsthm, amscd, color, graphicx,subcaption}

\newtheorem{theorem}{Theorem}

\newtheorem{lemma}{Lemma}

\newtheorem{defi}{Definition}

\newcommand{\ud}{\mathrm{d}}

\begin{document}
\title{Abrupt decorrelation for linear stochastic differential equations}

\maketitle

\begin{center}
{\large Sergio I. L\'opez,}\footnote{Departamento de Matem\'aticas, Facultad de Ciencias, UNAM, Av. Universidad 3000, Circuito Exterior s/n, Coyoac\'an,  C.P. 04510, Ciudad Universitaria, CDMX, Mexico.  E-mail: silo@ciencias.unam.mx }
{\large Juan C. Pardo}\footnote{CIMAT A.C. Calle Valenciana s/n, C.P. 36240, Guanajuato, Gto., Mexico. E-mail: jcpardo@cimat.mx}{\large
 \mbox{} and Leandro P. R. Pimentel}\footnote{Instituto de Matem\'atica, Universidade Federal do Rio de Janeiro, C.P. 68530, CEP 21941-909, Rio de Janeiro, RJ, Brasil. E-mail: leandro@im.ufrj.br}
\end{center}
\vspace{0.2in}

\begin{abstract} 
Understanding how  stochastic systems progressively loses memory of their initial state is a fundamental problem in probability and statistics. In this manuscript, we introduce the notion of {\it abrupt decorrelation}, which explicitly characterises a sharp and sudden loss of correlation over time. We study this phenomenon within a class of linear stochastic differential equations (LSDEs), for which explicit characterisations can be obtained under several statistical distances. Our analysis focuses primarily on the multivariate Ornstein-Uhlenbeck process, while in the one-dimensional setting we further extend the study to LSDEs with time-dependent drift coefficients. The results reveal strong analogies with the cut-off phenomenon for Markov processes and provide new insight into the mechanisms governing decorrelation in stochastic systems. 


\end{abstract}

\vspace{0.5cm}

\section{Introduction and main results.} 
Understanding how a stochastic system progressively loses  memory of its initial condition is a fundamental problem in probability and statistics. The rate and mechanism by which correlations decay over time play a central role in the analysis of stochastic models and arise naturally across a broad range of applications, including statistical physics, biology, finance, information theory, and the social sciences.

Beyond the mere existence of decorrelation, an equally important question concerns the time scales on which it occurs. Characteristic decorrelation times are essential for assessing the reliability of statistical inference procedures, the efficiency of simulation algorithms, and the validity of asymptotic approximations. More generally, the study of decorrelation provides a bridge between abstract probabilistic theory and practical applications by  offering quantitative tools for understanding the behaviour of complex random systems.

 A representative example appears in information theory through the additive white Gaussian noise (AWGN) channel, one of the fundamental models for communication systems, particularly in satellite and deep-space communications. In this setting, a signal is transmitted through successive stages while being affected by Gaussian noise and subject to input constraints. As the transmission evolves, information about the original message is progressively degraded, leading naturally to a weakening of the dependence between the current state of the system and the initial signal. Determining the time scale at which this dependence becomes negligible is therefore a problem of both theoretical and practical significance; see, for instance, Polyanskiy and Wu \cite{PolWu}.

Similar questions arise in statistics in the context of estimating stationary distributions from dependent observations. Even when stationarity holds only approximately, it is frequently adopted as a working assumption. Recent work by Lecestre \cite{Lec} demonstrates that the $\rho$-estimators introduced by  Baraud and collaborators in \cite{BBS,BB} remain remarkably robust under moderate dependence. In the setting of stationary hidden Markov models and Langevin diffusions, Lecestre showed that the information coefficient, measured through the Kullback-Leibler divergence between the joint distribution of a sample and the product of its marginals, decays exponentially fast. This exponential loss of dependence provides theoretical support for the effectiveness of $\rho$-estimators beyond the independent setting and again highlights the importance of identifying the relevant decorrelation time scales.

Related phenomena also appear in random matrix theory and interacting particle systems. A classical example is provided by the Gaussian Unitary Ensemble (GUE), whose eigenvalues form a strongly correlated random system. When one compares the spectra of a large matrix and those of its principal minors, correlations persist at nearby scales but gradually weaken as the size difference increases. In appropriate asymptotic regimes, one observes a transition between strongly correlated and essentially independent behaviour. Such results illustrate how decorrelation may emerge in high-dimensional stochastic systems through changes in scale, and they provide another perspective on the quantitative study of memory loss in complex random structures \cite{Ferr,Bao,Forr}.

A canonical example from statistical physics is provided by the motion of a microscopic particle suspended in a fluid and subjected simultaneously to friction and thermal fluctuations. The velocity of such a particle is classically modelled by the Langevin equation, whose dynamics are governed by an Ornstein-Uhlenbeck process. In this framework, the deterministic component represents viscous damping, while the stochastic forcing accounts for random molecular collisions. From a physical perspective, the Ornstein-Uhlenbeck process describes the relaxation of the system toward equilibrium under the persistent influence of randomness. In equilibrium, the process is stationary and its autocorrelation function decays exponentially fast, reflecting the gradual loss of memory induced by damping and noise. This behaviour is usually interpreted as a smooth relaxation mechanism, where correlations decrease continuously at a rate determined by the dissipation parameter.

However, a different picture may emerge in high-dimensional or many-body systems. When one considers collections of interacting or independent Ornstein-Uhlenbeck components, such as velocity fields, fluctuating modes in a medium, or effective degrees of freedom in coarse-grained models, the aggregate behaviour can exhibit a markedly sharper transition. In such situations, correlations may remain essentially unchanged over a substantial period and then decay abruptly within a comparatively short time window. Although each individual mode relaxes exponentially, the superposition of many modes with different relaxation rates can generate a threshold effect at the macroscopic level: before the threshold, sufficiently slow modes preserve the memory of the initial state, whereas beyond it their combined contribution rapidly becomes negligible.

The purpose of this manuscript is to introduce the notion of {\it abrupt decorrelation}, which formalises this sharp transition in the loss of correlation over time. We investigate this phenomenon within a class of linear stochastic differential equations (LSDEs), for which explicit calculations can be carried out under several statistical distances, thereby providing a transparent framework for analysis. In particular, we show that the unique strong solutions of these equations may exhibit a sudden loss of correlation, measured in an appropriate statistical distance, in a manner closely analogous to the well-known cut-off phenomenon for Markov processes, see for instance \cite{AlDi}.
\\

Our primary focus is on the multivariate Ornstein-Uhlenbeck (OU) process in $\mathbb{R}^d$, for $d\ge 1$. In the one-dimensional case, we extend our analysis to a broader class of LSDEs with time-dependent drift terms, a family that includes the OU process as a special case. Notably, abrupt decorrelation is not exclusive to LSDEs; similar behaviour has been observed in other randomly perturbed dynamical systems. For instance, the authors in \cite{BLPP1} investigate abrupt decorrelation, exclusively under the  total variation distance, for Langevin diffusions subjected to small random perturbations. 

Before formally introducing the concept of {\it abrupt decorrelation} and presenting our main results, we first introduce the random processes that underpins our analysis. Let $d, n\ge 1$ and consider a deterministic matrix $Q\in \mathbb{R}^{d\times d}$. We  assume that all eigenvalues of  $Q$  have strictly negative real part (i.e. Hurtwitz-stable matrix), ensuring that the  deterministic system 
\begin{equation}\label{difeq}
{\rm d}x_t=Qx_t {\rm d}t
\end{equation}
with initial condition $x_0=\xi\in \mathbb{R}^d$, admits a unique absolutely continuous solution defined for  all $ t\ge 0$. Moreover, this solution satisfies
\[
\lim_{t\to\infty} x_t=0,
\]
so  the origin  is an asymptotically stable equilibrium point. Closely related  is the matrix differential equation
\[
{\rm d}\Phi(t)=Q\Phi(t){\rm d}t, \qquad \Phi(0)={\bf I}_d,
\]
where ${\bf I}_d$ denotes the $d\times d$ identity matrix.  This equation also admits a unique, absolutely continuous solution $\Phi(t)$, defined for all $t\ge 0$. The matrix-valued function $\Phi$ is called the  {\it fundamental solution} to the deterministic equation   and  is non-singular for every $t\ge0$.

Building upon this deterministic framework, we  introduce the stochastic setting. Let $(\Omega, \mathcal{F}, \mathbb{P})$ be a probability space supporting an $n$-dimensional  Brownian motion $B=(B_t, t\ge 0)$ and a deterministic matrix $\sigma \in \mathbb{R}^{d\times n}$.  Let  $\xi$ be a  Gaussian random variable in $\mathbb{R}^d$ with mean $\mu_0\in \mathbb{R}^{d}$ and covariance matrix $\Sigma_0\in \mathbb{R}^{d\times d}$, independent of $B$. We consider the family of stochastic processes $(X^{(\epsilon)}; \epsilon>0)$ governed by the following stochastic differential equation (SDE):
\begin{equation}\label{SDE}
\ud X^{(\epsilon)}_t =QX^{(\epsilon)}_t \ud t +\epsilon \sigma \ud B_t\qquad \textrm{ for any }\quad t>0,
\end{equation}
with initial condition $X_0^{(\epsilon)}=\xi$. The  process $X^{(\epsilon)}=(X^{(\epsilon)}_t; t\ge 0)$ is known as a multivariate Ornstein-Uhlenbeck (OU) process parameterised by the noise intensity $\epsilon>0$.

Under our assumptions, the SDE admits a unique strong solution adapted to the natural filtration $(\mathcal{F}_t)_{t\ge 0}$ generated by $B$, satisfiyng  the usual conditions.  Furthermore, this solution  has the explicit representation:
\begin{equation}\label{specific}
X_{t}^{(\epsilon)}=e^{tQ}\left(X_0^{(\epsilon)} +\epsilon\int_0^t e^{-sQ}\sigma {\rm d} B_s\right) \qquad \textrm{for}\quad t\ge 0.
\end{equation}
The  process $X^{(\epsilon)}$ is  Gaussian  with mean
\[
m(t)=\mathbb{E}\Big[X_{t}^{(\epsilon)}\Big]=e^{tQ}\mu_0
\]
and covariance matrix, for $s, t\ge 0$, given by
\begin{equation}\label{covmatrix}
\begin{split}
	\rho(s,t)&=\mathbb{E}\left[\left(X^{(\epsilon)}_s-e^{sQ}\mu_0\right)\left(X_t^{(\epsilon)}-e^{tQ}\mu_0\right)^T\right]\\
	&=e^{sQ}\left(\Sigma_0+\epsilon^2\int_0^{s\land t}e^{-uQ}\sigma\sigma^T e^{-uQ^T}{\rm d}u\right)e^{tQ^T},
\end{split}
\end{equation}
where  $A^T$ denotes the transpose of a matrix $A$. We also denote $\rho_t:=\rho(t,t)$ for the variance (covariance matrix) of $X_t$.

To guarantee that the covariance matrix $\rho_t$ is non-singular for all $t>0$, we require the  controllability rank condition
\begin{equation}\label{rankcond}
\mathrm{Rank}\Big[\sigma, Q\sigma, Q^2\sigma, \ldots, Q^{d-1}\sigma\Big]=d,
\end{equation}
where $\Big[\sigma, Q\sigma, Q^2\sigma, \ldots, Q^{d-1}\sigma\Big]$ is the $d\times d n$ matrix formed by the horizontal concatenation of the indicated matrices. This matrix corresponds to the linear map 
\[
(u_0, \ldots, u_{d-1})\mapsto \sigma u_0+\cdots +Q^{d-1}\sigma u_{n-1}
\]
from $\mathbb{R}^{dn}$ into $\mathbb{R^d}.$ The controllability rank condition is  necessary and sufficient for the law of $X^{(\epsilon)}_t$ to admit  a smooth density in $\mathbb{R}^d$ with respect to the Lebesgue measure, see for instance the monographs of Da Prato and Zabczyk \cite{DPZAB} and Pavliotis \cite{Pav}. Moreover, under our conditions, i.e. that $Q$ is Hurtwitz-stable and that the controllability rank condition \eqref{rankcond} is fulfilled, $X^{(\epsilon)}$ converges in distribution, as $t$ grows to infinity, to a stationary (invariant) law with smooth density. The stationary distribution  is  also a centered Gaussian in $\mathbb{R}^d$ with covariance matrix \[
\rho_\infty:=\epsilon^2\int_0^{\infty}e^{uQ}\sigma\sigma^T e^{uQ^T}{\rm d}u\in \mathbb{R}^{d\times d}.
\]
In the one-dimensional setting, the situation becomes much simpler, allowing us to extend our study to a slight generalisation of the OU process, given by the following LSDE 
\eqref{SDE}:
\begin{equation}\label{LSDE}
\ud X^{(\epsilon)}_t=-(\theta -A(t)) X^{(\epsilon)}_t\ud t+\epsilon \sigma \ud B_t, 
\end{equation}
with $\theta, \sigma>0$ constants, $A:[0,\infty)\to\mathbb{R}$ a continuous function satisfying $A(t)\to 0$, as $t\to \infty$, and  $B=(B_t; t\ge 0)$ a standard Brownian motion. The initial condition $X^{(\epsilon)}_0=\xi$ is a  Gaussian random variable, independent of $B$,   with mean $\mu_0\in \mathbb{R}$ and variance $\sigma_0^2>0$. 

For our analysis, we  assume that the function $A:[0,\infty)\to \mathbb{R}$ satisfies
\begin{equation}\label{hyp1}
\int_0^\infty|A(t)|\ud t <\infty.
\end{equation}
Although this condition is not strictly necessary for the existence of a limiting distribution, since $\theta>0$ and the properties of  $A(t)$  already guarantee this, it is essential for the well-posedness and convergence of the decorrelation profile, which  plays a central role in our results.

In this setting, the fundamental solution of the  linear drift component is explicitly given by
\begin{equation}\label{eq:linear1}
\Phi(t)=e^{-\theta t+\int_0^t A(s)\ud s}, \qquad t\ge0,	
\end{equation}

which, by the above  assumption, is asymptotically equivalent to $\alpha e^{-\theta t}$, as $t\to \infty$ where
\[
\alpha:=e^{-\int_0^\infty A(s)\ud s}>0,
\]
which appears in the limiting behaviour of the covariance structure and the decorrelation profile. For further details we refer  the reader to Section 5.6 of Karatzas and Shreve \cite{KS91}, Section 5 of Da Prato and Zabczyk \cite{DPZAB} and Section 3.7 in Plaviotis \cite{Pav}.\\

Let $\mathcal{L}(Y)$ denote the law (or distribution) of a random variable $Y$. Our main result establishes that, as the noise intensity  $\epsilon$ tends to zero, the statistical dependence between the initial state $X_0^{(\epsilon)}$ and the state at time $t$, $X_t^{(\epsilon)}$, undergoes a sharp transition. Specifically, for an appropriate statistical distance that might depend on the parameter $\epsilon$, $d^{(\epsilon)}$, the difference between the joint  law  $\mathcal{L}(X_0^{(\epsilon)}, X_t^{(\epsilon)})$ and the product law $\mathcal{L}(X_0^{(\epsilon)})\otimes\mathcal{L}(X_t^{(\epsilon)})$,  sharply drops from its maximum  value to near zero within a narrow time window centred around a characteristic {\it decorrelation  time}.  Formally, this manifests as
\[
d^{(\epsilon)}\left( \mathcal{L}(X_0^{(\epsilon)}, X_t^{(\epsilon)}), \mathcal{L}(X_0^{(\epsilon)})\otimes\mathcal{L}(X_t^{(\epsilon)})\right)
\]
transitioning rapidly from high to low values as $t$ crosses a critical scale. In other words, prior to this window, the two states are strongly statistically dependent; shortly after, the dependence drops abruptly to near zero, and beyond this window, the convergence toward statistical independence occurs at an exponential rate. This phenomenon is what we refer to as  {\it abrupt decorrelation}. Furthermore, we show that the rescaled distance between the joint and product laws converges to a universal limiting curve as $\epsilon \to 0$. This limiting behaviour defines what we call the {\it decorrelation profile}, a canonical, parameter-free description of the transition from dependence to independence in the small-noise regime.

\subsection{Abrupt decorrelation.}

The concept of {\it abrupt decorrelation}  is inspired  by  the notion of  {\it cut-off } considered by  Barrera \&  Ycart  \cite{BY} which is expressed at three increasingly sharp levels and is defined for a parameterised  family of ergodic stochastic processes, similarly to the case discussed above. Following this approach, we introduce the notion of the {\it phenomenon of abrupt decorrelation} in three progressively refined levels. 

The most stringent concept is termed {\it decorrelation profile } which 
captures the exact asymptotic behaviour  of the collapse in the distance between the joint law of the system together with is initial condition and its  independent product law.
Decorrelation profile  implies a less restrictive notion known as {\it window decorrelation} where abrupt convergence occurs within a specific time interval, though the precise asymptotic profile is no longer captured. Finally, window decorrelation is generalised into the broader notion of {\it abrupt decorrelation},  which preserves the property of sharp convergence over time centred around a critical interval, but without quantifying the error.

More precisely let us consider $(X^{(\epsilon)}; \epsilon>0)$ a parametrised family of stochastic processes with  initial condition $X^{(\epsilon)}_0=\xi$. We denote by $\mathcal{L}_t^{(\epsilon)}$ for the law of the couple $(X_0^{(\epsilon)}, X_t^{(\epsilon)})$.   Since we are interested at times at which the system  $X^{(\epsilon)}$ may decorrelate from its initial condition, we introduce a r.v. $Y$ which is independent of $X^{(\epsilon)}$ but with the same distribution as $\xi$. We denote the law of the couple $(Y, X_t^{(\epsilon)})$ by $\mathcal{I}_t^{(\epsilon)}:=\mathcal{L}(Y)\otimes\mathcal{L}(X_t^{(\epsilon)})$. 

Let us introduce $d^{(\epsilon)}(t)$, the statistical distance between the distributions $\mathcal{L}_t^{(\epsilon)}$ and $\mathcal{I}_t^{(\epsilon)}$. Let us denoted by $M$ the diameter of the respective metric space of probability measures in which we are working in. In general, $M$ could be infinite.

\begin{defi}\label{abdeco}
We say that the family $(X^{(\epsilon)}; \epsilon>0):$ 
\begin{itemize}
	\item[i)] is {\bf abrupt decorrelated} at $(t_{\epsilon}; \epsilon>0)$  if $t_{\epsilon}$ goes to $+\infty$ when $\epsilon$ goes to $0$ and
	\begin{eqnarray*}
		\lim\limits_{\epsilon\rightarrow 0 }{d^{(\epsilon)}(ct_{\epsilon})}= \left\{ \begin{array}{lcc}
			M &  \textrm{ if }  & 0 < c < 1, \\
			\\ 0 & \textrm{ if } & c>1, \\
		\end{array}
		\right.
	\end{eqnarray*}
	\item[ii)] has a {\bf window decorrelation} at
	$\{\left(t_{\epsilon}, w_{\epsilon}\right); \epsilon>0\}$, if $t_{\epsilon}$ goes to $+\infty$ when $\epsilon$ goes to $0$, $w_{\epsilon}=o\left(t_{\epsilon}\right)$ and
	%\begin{itemize}
	%\item[$ii)$]
	\[
	\lim\limits_{r \rightarrow -\infty}{\liminf\limits_{\epsilon\rightarrow 0}
		{d^{(\epsilon)}(t_{\epsilon}+rw_{\epsilon})}}=M\qquad \textrm{and} \qquad  \lim\limits_{r \rightarrow +\infty}{\limsup\limits_{\epsilon\rightarrow 0}
		{d^{(\epsilon)}(t_{\epsilon}+rw_{\epsilon})}}=0,
	\]
	\item[iii)] has {\bf decorrelation profile}  at
	$\{\left(t_{\epsilon}, w_{\epsilon}\right); \epsilon>0\}$ with profile $G$, if $t_{\epsilon}$ goes to $+\infty$ when $\epsilon$ goes to $0$, $w_{\epsilon}=o\left(t_{\epsilon}\right)$,
	\begin{eqnarray*}
		G(r):=\lim\limits_{\epsilon \rightarrow 0}{d^{(\epsilon)}(t_{\epsilon}+rw_{\epsilon})}
	\end{eqnarray*} exists for all $r\in \mathbb{R}$ and
	\[
	\lim\limits_{r \rightarrow -\infty}{G(r)}=M\qquad \textrm{and} \qquad \lim\limits_{r \rightarrow +\infty}{G(r)}=0.
	\]
\end{itemize}
\end{defi}

Although the notions of cut-off and abrupt decorrelation may initially appear similar, they in fact capture fundamentally different phenomena. In our setting (OU processes), the cut-off phenomenon refers to the sharp convergence  of the system's law towards that of its invariant distribution which is somehow described in terms of the initial state and the asymptotic behaviour of $e^{Qt}$, see for instance \cite{BJ, BJ1}. In contrast, abrupt decorrelation pertains to the convergence of the covariance structure of the initial condition and its current state, specifically involving the rapid decay of correlations over a critical time window. These two notions operate on different statistical levels, i.e. mean versus second-order structure, and, as we will detail below, their interplay reveals rich dynamical behaviour that cannot be captured by one concept alone. This comparison between matrices introduces a higher level of analytical complexity, making the computations more involved.

\subsection{Statistical distances}
We next introduce the three statistical distances that will be used throughout this manuscript. Let $\mathbb{P}$ and $\mathbb{Q}$ be probability measures  defined in the same measurable space $\left(\Omega,\mathcal{F}\right)$.

The {\it total variation distance} between $\mathbb{P}$ and $\mathbb{Q}$ is defined by 
$$||\mathbb{P}-\mathbb{Q}||_{TV}:=\sup\limits_{A\in \mathcal{F}}{|\mathbb{P}(A)-\mathbb{Q}(A)|}.$$ 

Next, we recall the {\it  Kullback-Liebler divergence} or {\it relative entropy}.   If $\mathbb{P}$ is absolutely continuous with respect to $\mathbb{Q}$, we define
\[
D_{KL}(\mathbb{P}|\mathbb{Q}):=\int_{\Omega} \ln\left(\frac{\ud \mathbb{P}}{\ud \mathbb{Q}}\right) \ud \mathbb{P},
\] 
where $\frac{\ud \mathbb{P}}{\ud \mathbb{Q}}$ denotes the Radon-Nikodym derivative of $\mathbb{P}$  with respect to $\mathbb{Q}$. If $\mathbb{P}$ is not absolutely continuous with respect to $\mathbb{Q}$, then $D_{KL}(\mathbb{P}|\mathbb{Q})=\infty$. Equivalently, 
\[
D_{KL}(\mathbb{P}|\mathbb{Q})=\int_{\Omega} \frac{\ud \mathbb{P}}{\ud \mathbb{Q}}\ln\left(\frac{\ud \mathbb{P}}{\ud \mathbb{Q}}\right) \ud \mathbb{Q},
\] 
which represents the entropy of $\mathbb{P}$ relative to $\mathbb{Q}$. 

In the particular case of multivariate Gaussian distributions $\mathcal{N}$ and $\widetilde{\mathcal{N}}$ in $\mathbb{R}^d$, with respective means $\eta$ and $\widetilde{\eta}$, and covariance matrices $\Xi$ and $\widetilde{\Xi}$, the relative entropy admits the  explicit representation
\begin{equation}\label{KLgaus}
D_{KL}(\mathcal{N}|\widetilde{\mathcal{N}})=\frac{1}{2}\left(\mathrm{tr}\Big(\widetilde{\Xi}^{-1}\Xi\Big)-d+(\eta- \widetilde{\eta})^{T}\widetilde{\Xi}^{-1}(\eta-\widetilde{\eta})+\ln\left(\frac{\mathrm{det}\widetilde{\Xi}}{\mathrm{det}\Xi}\right)\right),
\end{equation}
see, for instance, Exercise 11 in Chapter 1 of \cite{Par}.

Finally, we introduce  {\it the Wasserstein distance of order 2} (also known as the  {\it Kantorovich-Rubinstein distance} or the {\it Fr\'echet distance}) defined by 
\begin{equation}\label{def:wasserstein}
\mathcal{W}_2(\mathbb{P},\mathbb{Q}):=
\inf_{\Pi} \left(\int_{\Omega\times \Omega}|u-v|^2\Pi(\ud u,\ud v)\right)^{1/2},
\end{equation}
where the infimum is taken over all couplings $\Pi$ of  $\mathbb{P}$ and $\mathbb{Q}$, that is, all probability measures on $\Omega\times \Omega$ having marginals $\mathbb{P}$ and $\mathbb{Q}$.

For multivariate Gaussian distributions $\mathcal{N}$ and $\widetilde{\mathcal{N}}$ in $\mathbb{R}^d$,  with  means $\eta$ and $\widetilde{\eta}$, and covariance matrices $\Xi$ and $\widetilde{\Xi}$, 
the Wasserstein distance of order $2$ admits the closed-form expression
\begin{equation}\label{W2gau}
\mathcal{W}_2^2(\mathcal{N},\widetilde{\mathcal{N}})=|| \eta -\widetilde{\eta}||_2^2+\mathrm{tr}\left(\Xi+\widetilde{\Xi}-2\left(\widetilde{\Xi}^{1/2}\Xi\widetilde{\Xi}^{1/2}\right)^{1/2}\right),
\end{equation}
see,  for example,  identity (4) in Dowson and Landau \cite{DL}. 
\subsection{Main results.}
For simplicity of exposition, we first consider the one-dimensional case.  In this setting,  the  family of LSDE processes $(X^{(\epsilon)}; \epsilon>0)$, defined in \eqref{LSDE}, exhibits explicit decorrelation profiles with respect to the  Kullback-Liebler divergence, the Wasserstein distance of order 2 and the total variation distance (see Figure \ref{Profile}). Notably, the form of these profiles 
is  explicit in each case.

\begin{theorem}\label{theorem1}  For the time scale 
\[
t_\epsilon:=\frac{|\ln \epsilon|}{\theta},
\]
the family of LSDE processes $(X^{(\epsilon)}; \epsilon>0)$ where for each $\epsilon>0$, $X^{(\epsilon)}$ is defined as in  \eqref{LSDE}, exhibits for all asymptotically constant window sizes $\omega_\epsilon \to \omega>0$,  decorrelation profile under: 
\begin{itemize}
	\item[i)] The Kullback-Liebler divergence with profile function
	\begin{eqnarray*}
		G_{KL}(r)=\lim_{\epsilon\to0}D_{KL}( \mathcal{L}^{(\epsilon)}_{t_\epsilon+r\omega_{\epsilon}} \Big|  \mathcal{I}^{(\epsilon)}_{t_\epsilon+r\omega_{\epsilon}} ) = \frac{1}{2} \ln \left(1+\frac{2\theta\sigma_0^2}{\sigma^2 \alpha^2e^{2\theta r\omega}}\right)\,;
	\end{eqnarray*}
	\item[ii)] The  normalised  Wasserstein distance of order 2 with profile function
	\[
	\begin{split}
		G_{\mathcal{W}_2}(r)&=\lim_{\epsilon\to0}\frac{\mathcal{W}_2\Big(\mathcal{L}^{(\epsilon)}_{t_\epsilon+r\omega_{\epsilon}},\mathcal{I}^{(\epsilon)}_{t_\epsilon+r\omega_{\epsilon}}\Big)}{\epsilon}=\sqrt{2\sigma^2_0\left(\frac{\sigma^2}{2\theta\sigma_0^2}+ \frac{e^{-2\theta r\omega}}{\alpha^2}-\sqrt{\frac{\sigma^4}{4\theta^2\sigma_0^4 }+\frac{\sigma^2e^{-2r\theta \omega}}{2\theta\sigma_0^2 \alpha^2}}\right)}\,;
	\end{split}
	\] 
	\item[iii)] The total variation distance with profile
	\[
	\begin{split}
		G_{TV}(r)&=\lim_{\epsilon\to0}||\mathcal{L}^{(\epsilon)}_{t_\epsilon+r\omega_{\epsilon}}-\mathcal{I}^{(\epsilon)}_{t_\epsilon+r\omega_{\epsilon}}||_{TV}=||\mathcal{N}(0, M(r) )-\mathcal{N}(0,\mathbf{I}_2)||_{TV},
	\end{split}
	\] 
	where  
	\[
	M(r)=\begin{pmatrix}
		1 & \phi(r)\\ \phi(r) & 1 \end{pmatrix}\qquad\textrm{and}\qquad \mathbf{I}_2=\begin{pmatrix}1 & 0\\ 0 & 1 \end{pmatrix},
	\]
	respectively, and $r\mapsto\phi(r)$ is given by
	\[
	\phi(r)=\sigma_0\sqrt{\frac{2\theta}{\sigma^2 \alpha^2e^{2\theta rw}+2\theta\sigma_0^2 }}.
	\]
	
\end{itemize}
\end{theorem}
It is important to note that the Kullback-Liebler divergence is not symmetric, in the sense that the limits of $D_{KL}( \mathcal{I}^{(\epsilon)}_{t} |  \mathcal{L}^{(\epsilon)}_{t} )$ and $D_{KL}( \mathcal{L}^{(\epsilon)}_{t}  |  \mathcal{I}^{(\epsilon)}_{t})$ provide different profile functions. Indeed \[
G_{KL}(r)=\lim_{\epsilon\to0}D_{KL}( \mathcal{I}^{(\epsilon)}_{t_\epsilon+r\omega_{\epsilon}} |  \mathcal{L}^{(\epsilon)}_{t_\epsilon+r\omega_{\epsilon}}    )=\frac{ \theta  \sigma^2_0}{\sigma^2 \alpha^2e^{2\theta r\omega}} -  \frac{1}{2} \ln \left(1+\frac{2\theta\sigma_0^2}{\sigma^2 \alpha^2e^{2\theta r\omega}}\right),
\]
which is different to the profile function in Theorem \ref{theorem1}.\\

\begin{figure}
\begin{subfigure}{.3\textwidth}
	\centering
	\includegraphics[width=.8\linewidth]{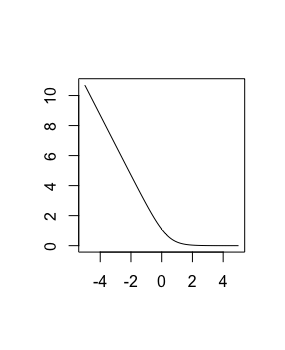}
	\caption{Kullback-Liebler}
	\label{fig:sfig1}
\end{subfigure}
\begin{subfigure}{.3\textwidth}
	\centering
	\includegraphics[width=.8\linewidth]{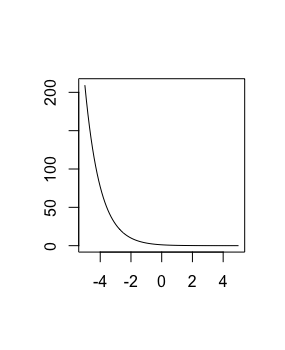}
	\caption{Wasserstein of order $2$}
	\label{fig:sfig2}
\end{subfigure}
\begin{subfigure}{.3\textwidth}
	\centering
	\includegraphics[width=.8\linewidth]{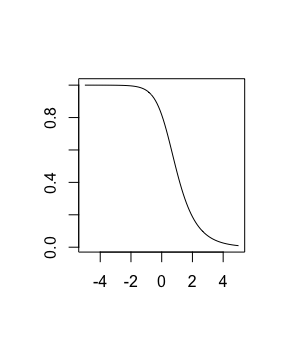}
	\caption{Total Variation}
	\label{fig:sfig3}
\end{subfigure}
\caption{Plots of decorrelation profiles with $\theta=\sigma^2=\sigma_0^2=w=\alpha=1$. }
\label{Profile}
\end{figure}

We now focus on  the case $d\ge 2$. Before presenting our second and final main result, we introduce some notation. Let $P$ denote the change-of-basis matrix in the Jordan decomposition of $Q$, so that $Q=PBP^{-1}$ where $B$ is the Jordan normal form of $Q$ which can be written as
\[
B=\textrm{diag}(B_1,\ldots, B_p),
\]
in terms of the Jordan blocks $B_i$, for $1\le i\le p$. Each $B_i$ is a square matrix associated with a  single eigenvalue.

We denote the eigenvalues of $Q$  by $(\lambda_i)_{1\le i\le p}$ and set 
\[
\vartheta:=\max \{ \Re(\lambda_j)\,:\,j=1,\dots,p \}.
\]

Under our assumptions, we have $\vartheta<0$. An eigenvalue with real part  equal to  $\vartheta$ is called  a critical or leading eigenvalue, which we denote by $\lambda^*$. Let $m$ denote the maximal size of the Jordan blocks associated with any of the leading eigenvalues $\lambda^*$'s and let $\ell$ denote the number of blocks of maximal size $m$.

For $n\in \mathbb{N}$, we define  $E^{(n)}_{ij}$ as the $n\times n$ matrix with a 1 in the $(i,j)$-th position and zeros elsewhere. In particular, 
\[
E^{(n)}_{1n}:=\begin{pmatrix} 0&\dots &0& 1\\0&\dots &0& 0\\
\vdots&  &\vdots& \vdots\\ 0&\dots& 0&0\end{pmatrix}.
\]
The covariance matrix $\rho_t$ of the OU process satisfies
\begin{equation}\label{covt}
\rho_t =e^{tQ}\left(\Sigma_0+\epsilon^2\int_0^{t}e^{-uQ}\sigma\sigma^T e^{-uQ^T}{\rm d}u\right)e^{tQ^T}
, \qquad\textrm{for}\quad t\ge 0.
\end{equation}
For convenience, we also introduce
\begin{equation}\label{vars}
\varsigma_t := \int_{0}^{t} 	e^{sQ} \sigma \sigma^T e^{ s Q^{T} } {\rm d}s .
\end{equation}
Under our assumptions,  the limit $ \varsigma_{\infty} := \lim_{ t \rightarrow \infty} \varsigma_t$ exists, see for instance \cite{KS91}.

In this regime, we show that the family of OU process $(X^{(\epsilon)}; \epsilon>0)$, defined by the SDE in \eqref{SDE}, exhibits explicit decorrelation profiles with respect to both the Kullback-Liebler divergence and the total variation distance. It is also plausible that such a profile could be established under the Wasserstein distance of order 2; however, deriving the corresponding closed-form expression 
appears considerably more challenging.

More precisely, when there is a unique real leading eigenvalue $\lambda^*$, the  OU process exhibits a clear and quantifiable transition in its correlation structure. In contrast, if some of the leading eigenvalues $\lambda^*$'s are complex, the decorrelation behaviour is more intricate: the process displays window decorrelation but not decorrelation profile phenomenon, a consequence driven by the oscillatory nature of the underlying dynamics. In the following we will refer as the real case and the complex one, to the aforementioned 
possibilities concerning to the leading eigenvalues.

\begin{theorem}\label{theorem2}  
Consider the family of OU processes $(X^{(\epsilon)}; \epsilon>0)$, defined in \eqref{SDE} and assume that  $Q$ is Hurwitz stable and the controllability rank condition is fulfilled. Then, abrupt decorrelation occurs with respect to both the  Kullback-Liebler divergence and in the total variation distance at the time scale
$$  t_\epsilon:=\frac{1}{|\vartheta|}\Big(|\ln \epsilon|+\left(m-1\right)\ln|\ln \epsilon|\Big).$$  Moreover,
\begin{itemize}
	\item[i)] when we are in the complex case of the leading eigenvalues $\lambda^*$'s, the family $(X^{(\epsilon)}; \epsilon>0)$ exhibits
	window decorrelation with asymptotically constant window sizes $w_\epsilon\to \frac{w}{|\vartheta|}>0$ in both Kullback-Leibler divergence and total variation distance;
	
	\item[ii)] in the case where there is a unique real eigenvalue $\lambda^*$, the family $\{X^{(\epsilon)}; \epsilon>0\}$ exhibits decorrelation profile phenomenon with  asymptotically constant window sizes $w_\epsilon\to \frac{w}{|\vartheta|}>0$ under; 
	\begin{enumerate}
		\item[a)] The Kullback-Leibler divergence with profile function
		\[
		\begin{split}
			G_{KL}(r) &= \lim_{\epsilon\to0}D_{KL}( \mathcal{L}^{(\epsilon)}_{t_\epsilon+r\omega_{\epsilon}} \Big|  \mathcal{I}^{(\epsilon)}_{t_\epsilon+r\omega_{\epsilon}} )  \\
			& =   -\frac{1}{2}\ln\left(\det\left( \mathbf{I}_d- \Gamma\Sigma_0\Gamma^T\left(\Gamma\Sigma_0\Gamma^T+ e^{2rw}\varsigma_{\infty} \right)^{-1}\right) \right)\,,
		\end{split}
		\]
		where 
		\[
		\Gamma=\frac{|\vartheta|^{1-m}}{(m-1)!} P_1\mathbf{I}^{\ell}_mP_1^{-1}, 
		\] 
		the matrix $P_1$ is obtained by rearranging the block order of $P$ and 
		\[
		\mathbf{I}^{\ell}_m=\begin{pmatrix}
			E_{1m}^{(m)}& & & & &\\ 
			&\ddots & & & &\\
			& &E_{1m}^{(m)} & & &\\
			& & & 0& & \\
			& & & &\ddots & \\
			& & & & & 0
		\end{pmatrix},
		\]
		with $\ell\geq 1$ denoting the number of times that $E_{1m}^{(m)}$ appears along the diagonal of $\mathbf{I}^{\ell}_m$;
		\item[b)] The total variation distance with profile function
		$$G_{TV}(r):=\lim_{\epsilon\to0}||\mathcal{L}^{(\epsilon)}_{t_\epsilon+r\omega_{\epsilon}}-\mathcal{I}^{(\epsilon)}_{t_\epsilon+r\omega_{\epsilon}}||_{TV} = ||\mathcal{N}(0, M(r) )-\mathcal{N}(0,\mathbf{I}_{2d} )||_{TV}, $$ 
		where $M(r)$ is a block matrix given by 
		\[
		M(r)=\begin{pmatrix}
			\mathbf{I}_d &  \Sigma_0^{\frac{1}{2}} \Gamma^T  ( \Gamma \Sigma_0 \Gamma^T + e^{2r w} \varsigma_{\infty} )^{-\frac{1}{2}} 
			\\ ( \Gamma \Sigma_0 \Gamma^T + e^{2rw} \varsigma_{\infty} )^{-\frac{1}{2}} \Gamma \Sigma_0^{\frac{1}{2}} & \mathbf{I}_d  \end{pmatrix}\,.
		\]
	\end{enumerate} 
\end{itemize}

\end{theorem}

There is a natural explanation for the coincidence between the cut-off times \cite{BJ, BJ1} and the abrupt decorrelation times. From the explicit representation of the solution \eqref{specific}, one observes that the contribution of the initial condition and that of the diffusive term are separated and mutually independent. Both terms are multiplied by the factor $e^{tQ}$, which converges exponentially fast to zero as $t \to \infty$. Hence, in the long-time regime, the influence of the initial condition vanishes, while the diffusive term converges to the stationary distribution (whenever it exists). The transition occurs precisely at the time scale at which the exponential decay of the initial condition becomes effective, thereby explaining why the cut-off phenomenon and abrupt decorrelation coincide.

We also note that the initial condition may be chosen from an arbitrary distribution, not necessarily Gaussian. In such cases, the OU process is no longer Gaussian, and the explicit forms for the decorrelation profiles with respect to the Kullback-Leibler divergence, the Wasserstein distance of order 2, and the total variation distance are lost. Nonetheless, the qualitative mechanism remains visible, at least under the total variation distance (see, for instance \cite{BLPP1}). We also expect the same phenomenon to persist when the drift term is time-dependent in higher dimensions ($d \geq 2$). Yet, even in the deterministic case, the asymptotic behaviour of inhomogeneous linear systems is subtle and far from being fully understood.

The remainder of this   manuscript is devoted to the proofs of Theorems \ref{theorem1} and \ref{theorem2}.

\section{Proofs}
In this section, we provide the proofs of our main results. For simplicity on exposition,  we first deal with the one dimensional case.
\subsection{One dimensional case.} 
Recall that, for $t\ge 0$, the law of $X^{(\epsilon)}_t$, defined by \eqref{LSDE}, is Gaussian with  mean and variance
\[
\mathbb{E}\Big[X^{(\epsilon)}_t\Big]=\mu_0\Phi(t)\qquad \textrm{and}\qquad \mathrm{Var}\Big(X^{(\epsilon)}_t\Big)=\epsilon^2\sigma^2\Phi^2(t)\int_0^{ t}\frac{{\rm d}u}{\Phi^2(u)}+\sigma_0^2\Phi^2(t),
\]
where $\Phi$ is given by \eqref{eq:linear1}. Moreover the covariance between $X^{(\epsilon)}_0$ and $X^{(\epsilon)}_t$ is such that
\[
\begin{split}
	\mathrm{Cov}\Big(X^{(\epsilon)}_t, X^{(\epsilon)}_0\Big)=\Phi(t)\mathrm{Var}(X^{(\epsilon)}_0)=\Phi(t)\sigma^2_0.
\end{split}
\]
In other words, the couple $(X^{(\epsilon)}_0, X^{(\epsilon)}_t)$  has a Gaussian distribution with mean  $\eta_t=(\mu_0, \mu_0\Phi(t))$ and covariance matrix
\[
\Xi_t=\begin{pmatrix}
	\sigma_0^2 & \sigma_0^2\Phi(t)\\
	\sigma_0^2\Phi(t) & \epsilon^2\sigma^2\Phi^2(t)\int_0^{ t}\frac{{\rm d}u}{\Phi^2(u)}+\sigma_0^2\Phi^2(t)
\end{pmatrix}
\]
and the couple $(Y, X^{(\epsilon)}_t)$ has Gaussian distribution with mean   $\widetilde{\eta}_t=(\mu_0, \mu_0\Phi(t))$ and covariance matrix
\[
\widetilde{\Xi}_t=\begin{pmatrix}
	\sigma_0^2 & 0\\
	0 & \epsilon^2\sigma^2\Phi^2(t)\int_0^{ t}\frac{{\rm d}u}{\Phi^2(u)}+\sigma_0^2\Phi^2(t)
\end{pmatrix} .
\]
\begin{proof}[Proof of Theorem \ref{theorem1}] We first treat part (i). Since $ (X^{(\epsilon)}_0, X^{(\epsilon)}_t)$ and $ (Y, X^{(\epsilon)}_t)$ are multivariate Gaussian distributions, from \eqref{KLgaus}, we have
	\[
	D_{KL}\Big(\mathcal{L}^{(\epsilon)}_t\Big|\mathcal{I}^{(\epsilon)}_t\Big)= \frac{1}{2}\left(\mathrm{tr}(\widetilde{\Xi}_{t}^{-1}\Xi_{t})-2+(\eta_{t}-\widetilde{\eta}_{t})^{T}\widetilde{\Xi}_{t}^{-1}(\eta_{t}-\widetilde{\eta}_{t})+\ln\left(\frac{\mathrm{det }\widetilde{\Xi}_{t}}{\mathrm{det }\Xi_{t}}\right)\right).
	\]
	Observe that 
	\[
	\mathrm{det}\Xi_{t}=\epsilon^2\sigma^2\sigma_0^2\Phi^2(t)\int_0^{ t}\frac{{\rm d}u}{\Phi^2(u)} \qquad \textrm{and}\qquad \mathrm{det}\widetilde{\Xi}_{t}=\epsilon^2\sigma^2\sigma_0^2\Phi^2(t)\int_0^{ t}\frac{{\rm d}u}{\Phi^2(u)}+\sigma_0^4\Phi^2(t).
	\]
	Thus
	\[
	\ln \left(\frac{\mathrm{det}\widetilde{\Xi}_{t}}{\mathrm{det}\Xi_{t}}\right)=\ln\left(1+\frac{\sigma_0^2}{\epsilon^2\sigma^2\int_0^{ t}\frac{{\rm d}u}{\Phi^2(u)}}\right).
	\]
	On the other hand
	\[
	\widetilde{\Xi}_{t}^{-1}=\frac{1}{\mathrm{det}\widetilde{\Xi}_{t}}
	\begin{pmatrix}
		\epsilon^2\sigma^2\Phi^2(t)\int_0^{ t}\frac{{\rm d}u}{\Phi^2(u)}+\sigma_0^2\Phi^2(t)
		& 0\\
		0 & \sigma_0^2
	\end{pmatrix},
	\]
	implying that 
	\[
	\begin{split}
		\mathrm{tr}(\widetilde{\Xi}_{t}^{-1}\Xi_{t})&=\frac{1}{\mathrm{det}\widetilde{\Xi}_{t}}\left(2\sigma^2_0\left[\epsilon^2\sigma^2\Phi^2(t)\int_0^{ t}\frac{{\rm d}u}{\Phi^2(u)}+\sigma_0^2\Phi^2(t)\right]\right)=2.
	\end{split}
	\]
	In other words,
	\[
	D_{KL}\Big(\mathcal{L}^{(\epsilon)}_t\Big|\mathcal{I}^{(\epsilon)}_t\Big)=\frac{1}{2}\ln\left(1+\frac{\sigma_0^2}{\epsilon^2\sigma^2\int_0^{ t}\frac{{\rm d}u}{\Phi^2(u)}}\right).
	\]
	Next we observe that for $t$ sufficiently large,
	\begin{equation}\label{asymptotics}
		e^{\theta t}\Phi(t)= e^{\int_0^t A(s)\ud s}\sim \frac{1}{\alpha}  \quad \textrm{and}\quad \Phi^2(t)\int_0^{ t}\frac{{\rm d}u}{\Phi^2(u)}=\int_0^t e^{-2\theta (t-u)+2\int_u^{t} A(s){\rm d} s}{\rm d} u\sim \frac{1}{2\theta}. 
	\end{equation}
	The latter implies that there is abrupt decorrelation, in the following sense. First, we take $c>0$ and $t_\epsilon=\frac{1}{\theta}|\ln \epsilon|$, for $\epsilon$ small enough, and observe
	\[
	e^{2\theta c t_\epsilon}=e^{2 c \log(1/\epsilon)}=\epsilon^{-2 c},
	\] 
	which implies
	\[
	\epsilon^2e^{2\theta c t_\epsilon}=\epsilon^2\epsilon^{-2 c}=\epsilon^{2(1-c)}\xrightarrow[\epsilon\to 0]{}\left\{
	\begin{array}{ll}
		0,  & \mathrm{ if } \,\,c<1,\\
		+\infty, & \mathrm{ if }\,\, c>1.
	\end{array}
	\right.
	\]
	In other words,  there is abrupt decorrelation for $\{X^{(\epsilon)}, \epsilon>0\}$  at times $t_\epsilon$, since for $c>1$
	\[
	D_{KL}\Big(\mathcal{L}^{(\epsilon)}_{c t_\epsilon}\Big|\mathcal{I}^{(\epsilon)}_{c t_\epsilon}\Big)\xrightarrow[\epsilon\to 0]{}0,
	\]
	and for  $c<1$, we have 
	\[
	D_{KL}\Big(\mathcal{L}^{(\epsilon)}_{c t_\epsilon}\Big|\mathcal{I}^{(\epsilon)}_{c t_\epsilon}\Big)\xrightarrow[\epsilon\to 0]{}+\infty,
	\]
	given that
	\[
	\begin{split}
		\lim_{\epsilon\to 0}\frac{1}{2}\ln\left(1+\frac{\sigma_0^2}{\epsilon^2\sigma^2\int_0^{ c t_\epsilon}\frac{{\rm d}u}{\Phi^2(u)}}\right)&=\lim_{\epsilon\to 0}\frac{1}{2}\ln\left(1+\frac{2\theta\sigma_0^2}{\epsilon^2\sigma^2 \alpha^2e^{2\theta c t_\epsilon}}\right)=\lim_{z\to 0}\frac{1}{2}\ln\left(1+\frac{2\theta\sigma_0^2}{z}\right)=\infty,
	\end{split}
	\]
	where we have used the asymptotic behaviours  in \eqref{asymptotics}. Finally, we observe that $\{X^{(\epsilon)}, \epsilon>0\}$ has  decorrelation profile. In other words, take $\omega_\epsilon\to \omega$, as $\epsilon\to 0$, then from \eqref{asymptotics}, we have
	\[
	\begin{split}
		G_{KL}(r)&=\lim_{\epsilon\to0}D_{KL}(\mathcal{L}^{(\epsilon)}_{t_\epsilon+r\omega_{\epsilon}}|\mathcal{I}^{(\epsilon)}_{t_\epsilon+r\omega_{\epsilon}})=\lim_{\epsilon\to0}\frac{1}{2}\ln\left(1+\frac{\sigma_0^2}{\epsilon^2\sigma^2\int_0^{ t_\epsilon+r\omega_{\epsilon}}\frac{{\rm d}u}{\Phi^2(u)}}\right)\\
		&= \lim_{\epsilon\to0}\frac{1}{2}\ln\left(1+\frac{2\theta\sigma_0^2}{\sigma^2 \alpha^2e^{2\theta r\omega_\epsilon}}\right)=\frac{1}{2}\ln\left(1+\frac{2\theta\sigma_0^2}{\sigma^2 \alpha^2e^{2\theta r\omega}}\right),
	\end{split}
	\]
	as expected. This prove part (i).
	
	Now, we treat part (ii). For simplicity of exposition,  we introduce the following notation. Let $a:=\sigma_0^2, c_t:=\sigma_0^2\Phi(t)$ and 
	\begin{equation}\label{notpartii}
		b_t:=\epsilon^2\sigma^2\Phi^2(t)\int_0^{ t}\frac{{\rm d}u}{\Phi^2(u)}+\sigma_0^2\Phi^2(t), \qquad \textrm{for}\quad t\ge 0.
	\end{equation}
	In other words, we have
	\begin{equation}\label{matrices}
		\Xi_t =\begin{pmatrix}
			a & c_t\\
			c_t & b_t
		\end{pmatrix} \qquad\textrm{ and }\qquad\widetilde{\Xi}_t =\begin{pmatrix}
			a & 0\\
			0 & b_t
		\end{pmatrix}.
	\end{equation}
	Thus from \eqref{W2gau}, we obtain
	\[
	\mathcal{W}^2_2\Big(\mathcal{L}^{(\epsilon)}_{t},\mathcal{I}^{(\epsilon)}_{t}\Big)=\mathrm{tr}\left(\Xi_t+\widetilde{\Xi}_t-2\left(\widetilde{\Xi}^{1/2}_t\Xi_t\widetilde{\Xi}_t^{1/2}\right)^{1/2}\right).
	\]
	First, we observe that 
	\[
	\mathrm{tr}\Big(\Xi_t+\widetilde{\Xi}_t\Big)=2(a+b_t) \qquad \textrm{ and }\qquad \widetilde{\Xi}^{1/2}_t =\begin{pmatrix}
		a^{1/2} & 0\\
		0 & b^{1/2}_t
	\end{pmatrix}.
	\]
	Thus
	\begin{equation}\label{tomate}
		\begin{split}
			\widetilde{\Xi}_t^{1/2}\Xi_t\widetilde{\Xi}_t^{1/2}&=\begin{pmatrix}
				a^{1/2} & 0\\
				0 & b^{1/2}_t
			\end{pmatrix}
			\begin{pmatrix}
				a & c_t\\
				c_t & b_t
			\end{pmatrix}
			\begin{pmatrix}
				a^{1/2} & 0\\
				0 & b^{1/2}_t
			\end{pmatrix}\\
			&=\begin{pmatrix}
				a^{3/2} & a^{1/2}c_t\\
				b^{1/2}_tc_t & b^{3/2}_t
			\end{pmatrix}
			\begin{pmatrix}
				a^{1/2} & 0\\
				0 & b^{1/2}_t
			\end{pmatrix}\\
			&=\begin{pmatrix}
				a^{2} & (ab_t)^{1/2}c_t\\
				(ab_t)^{1/2}c_t & b^{2}_t
			\end{pmatrix}.
		\end{split}
	\end{equation}
	In order to compute the trace of the above matrix, we first note that its eigenvalues   satisfy
	\[
	\lambda_{\pm}(t):=\frac{1}{2}\left(a^2+b_t^2\pm \sqrt{(a^2-b^2_t)^2+4ab_tc_t^2}\right),
	\]
	thus the  trace of $\widetilde{\Xi}_t^{1/2}\Xi_t\widetilde{\Xi}_t^{1/2}$ is such that 
	\[
	\mathrm{tr}\left(\left(\widetilde{\Xi}_t^{1/2}\Xi_t\widetilde{\Xi}_t^{1/2}\right)^{1/2}\right)=\lambda^{1/2}_+(t) +\lambda^{1/2}_{-}(t).
	\]
	Putting all pieces together, we have
	\[
	\mathcal{W}^2_2\Big(\mathcal{L}^{(\epsilon)}_{t},\mathcal{I}^{(\epsilon)}_{t}\Big)=2\left(a+b_t-\sqrt{\frac{1}{2}\left(a^2+b^2_t+ \sqrt{\Delta(t)}\right)}-\sqrt{\frac{1}{2}\left(a^2+b^2_t- \sqrt{\Delta(t)}\right)}\right),
	\]
	where $\Delta(t):=(a^2-b^2_t)^2+4ab_tc^2_t$. 
	
	Next for $\epsilon\in(0,1)$ and $\omega\in \mathbb{R}$, we define
	$$ \overline{t}_\epsilon:=\frac{|\ln \epsilon|}{\theta} +\omega$$ 
	and observe 
	\[
	\begin{split}
		\Delta(\overline{t}_\epsilon)&=\sigma_0^8\left[\left(1-\left(\frac{\epsilon^2\sigma^2}{\sigma_0^2}\Phi^2(\overline{t}_\epsilon)\int_0^{ \overline{t}_\epsilon}\frac{{\rm d}u}{\Phi^2(u)}+\Phi^2(\overline{t}_\epsilon)\right)^2\right)^2\right.\\
		&\hspace{1cm}\left. +4 \Phi^2(\overline{t}_\epsilon)\left(\frac{\epsilon^2\sigma^2}{\sigma_0^2}\Phi^2(\overline{t}_\epsilon)\int_0^{ \overline{t}_\epsilon}\frac{{\rm d}u}{\Phi^2(u)}+\Phi^2(\overline{t}_\epsilon)\right)\right]\\
		&=\sigma_0^8\left[\left(1-\epsilon^4\left(\frac{\sigma^2}{\sigma_0^2}\Phi^2(\overline{t}_\epsilon)\int_0^{ \overline{t}_\epsilon}\frac{{\rm d}u}{\Phi^2(u)}+e^{-2\theta \omega}e^{2\theta \overline{t}_\epsilon}\Phi^2(\overline{t}_\epsilon) \right)^2\right)^2\right.\\
		&\hspace{1cm}\left. +4  \epsilon^4e^{-2\theta \omega}e^{2\theta \overline{t}_\epsilon}\Phi^2(\overline{t}_\epsilon)\left(\frac{\sigma^2}{\sigma_0^2}\Phi^2(\overline{t}_\epsilon)\int_0^{ \overline{t}_\epsilon}\frac{{\rm d}u}{\Phi^2(u)}+e^{-2\theta \omega}e^{2\theta \overline{t}_\epsilon}\Phi^2(\overline{t}_\epsilon)\right)\right]\\
		&=\sigma_0^8\left[\left(1-\epsilon^4 g_\epsilon(\omega)^2\right)^2+4 \epsilon^4e^{-2\theta \omega}e^{2\theta \overline{t}_\epsilon}\Phi^2(\overline{t}_\epsilon)g_\epsilon(\omega)\right]
	\end{split}
	\]
	where 
	\[
	g_\epsilon(\omega):=\frac{\sigma^2}{\sigma_0^2}\Phi^2(\overline{t}_\epsilon)\int_0^{ \overline{t}_\epsilon}\frac{{\rm d}u}{\Phi^2(u)}+e^{-2\theta \omega}e^{2\theta \overline{t}_\epsilon}\Phi^2(\overline{t}_\epsilon).
	\]
	Hence
	\[
	\begin{split}
		\frac{\Delta(\overline{t}_\epsilon)}{\sigma_0^8}
		&=1-2\epsilon^4 g_\epsilon(\omega)^2+\epsilon^8 g_\epsilon(\omega)^4+4 \epsilon^4e^{-2\theta \omega}e^{2\theta \overline{t}_\epsilon}\Phi^2(\overline{t}_\epsilon)g_\epsilon(\omega)\\
		&=1-\epsilon^4 g_\epsilon(\omega)\Big(2g_\epsilon(\omega)-\epsilon^4 g_\epsilon(\omega)^3-4 e^{-2\theta \omega}e^{2\theta \overline{t}_\epsilon}\Phi^2(\overline{t}_\epsilon)\Big)\\
		&=1-\epsilon^4 h_\epsilon(\omega),
	\end{split}
	\]
	where $h_\epsilon(\omega):=g_\epsilon(\omega)(2g_\epsilon(\omega)-\epsilon^4 g_\epsilon(\omega)^3-4 e^{-2\theta \omega}e^{2\theta \overline{t}_\epsilon}\Phi^2(\overline{t}_\epsilon))$.
	We note that $g_\epsilon(\omega)$ and $h_\epsilon(\omega)$ are continuous functions of $\epsilon\in[0,1)$, and at $\epsilon=0$, after using \eqref{asymptotics},  we have
	\begin{equation}\label{zero}
		g_0(\omega)=\frac{\sigma^2}{2\theta\sigma_0^2}+ \frac{e^{-2\theta \omega}}{\alpha^2}\qquad\mbox{ and }\qquad h_0(\omega)=g_0(\omega)\Big(2g_0(\omega)-4 \frac{e^{-2\theta \omega}}{\alpha^2}\Big).
	\end{equation}
	Recalling that $\sqrt{1+\delta}=1+\frac{\delta}{2}+O(\delta^2)$ for $\delta$ small enough, we deduce
	$$\sqrt{\frac{\Delta(\overline{t}_\epsilon)}{\sigma_0^8}}=\sqrt{1-\epsilon^4 h_\epsilon(\omega)}=1-\epsilon^4\frac{h_\epsilon(\omega)}{2}+O(\epsilon^8), $$
	for $\epsilon$ close to $0$. Since 
	$$\frac{a^2+b^2_{\overline{t}_\epsilon}}{\sigma_0^4}=1+\epsilon^4g_\epsilon(\omega)^2\,,$$
	we get that 
	$$\frac{a^2+b^2_{\overline{t}_\epsilon}+\sqrt{\Delta(\overline{t}_\epsilon)}}{\sigma_0^4}=2+\epsilon^4\left( g_\epsilon(\omega)^2-\frac{h_\epsilon(\omega)}{2}\right) +O(\epsilon^8)\,,$$
	and 
	$$\frac{a^2+b^2_{\overline{t}_\epsilon}-\sqrt{\Delta(\overline{t}_\epsilon)}}{\sigma_0^4}=\epsilon^4\left( g_\epsilon(\omega)^2+\frac{h_\epsilon(\omega)}{2}\right) +O(\epsilon^8)=\epsilon^4\left(  g_\epsilon(\omega)^2+\frac{h_\epsilon(\omega)}{2} +O(\epsilon^4)\right)\,.$$
	Hence,
	$$\sqrt{\frac{a^2+b^2_{\overline{t}_\epsilon}+\sqrt{\Delta(\overline{t}_\epsilon)}}{2\sigma_0^4}}=\sqrt{ 1+\frac{\epsilon^4}{2}\left( g_\epsilon(\omega)^2-\frac{h_\epsilon(\omega)}{2}\right)  +O(\epsilon^8)}=1+\frac{\epsilon^4}{4}\left( g_\epsilon(\omega)^2-\frac{h_\epsilon(\omega)}{2}\right) +O(\epsilon^8)\,,$$
	and 
	$$\sqrt{\frac{a^2+b^2_{\overline{t}_\epsilon}-\sqrt{\Delta(\overline{t}_\epsilon)}}{2\sigma_0^4}}=\sqrt{\frac{\epsilon^4}{2}\left(g_\epsilon(\omega)^2+\frac{h_\epsilon(\omega)}{2}+O(\epsilon^4)\right)}=\epsilon^2\sqrt{ \frac{1}{2}\left( g_\epsilon(\omega)^2+\frac{h_\epsilon(\omega)}{2}\right) +O(\epsilon^4)}\,.$$
	Using the previous estimates, we can conclude that 
	\[
	\begin{split}
		\mathcal{W}^2_2\Big(\mathcal{L}^{(\epsilon)}_{\overline{t}_\epsilon},\mathcal{I}^{(\epsilon)}_{\overline{t}_\epsilon}\Big)&=2\left(a+b_{t_\epsilon}-\sqrt{\frac{1}{2}\left(a^2+b^2_{\overline{t}_\epsilon}+ \sqrt{\Delta(\overline{t}_\epsilon)}\right)}-\sqrt{\frac{1}{2}\left(a^2+b^2_{\overline{t}_\epsilon}- \sqrt{\Delta(\overline{t}_\epsilon)}\right)}\right)\\
		&=2\sigma^2_0\left(1+\epsilon^2 g_\epsilon(\omega)-\left(1+\frac{\epsilon^4}{4}\left( g_\epsilon(\omega)^2-\frac{h_\epsilon(\omega)}{2}\right)+O(\epsilon^8)\right)\right.\\
		&\left.\hspace{6cm}- \epsilon^2\sqrt{ \frac{1}{2}\left(  g_\epsilon(\omega)^2+\frac{h_\epsilon(\omega)}{2}\right) +O(\epsilon^4)}\,\right)\\ 
		&=2\sigma^2_0\epsilon^2\left(g_\epsilon(\omega)- \sqrt{ \frac{1}{2}\left( g_\epsilon(\omega)^2+\frac{h_\epsilon(\omega)}{2}\right) +O(\epsilon^4)} \right.\\
		&\left.\hspace{6cm}- \frac{\epsilon^2}{4}\left( g_\epsilon(\omega)^2-\frac{h_\epsilon(\omega)}{2}\right) +\frac{O(\epsilon^8)}{\epsilon^2}\right)\,. 
	\end{split} 
	\]
	Together with \eqref{zero}, this implies that 
	\[
	\begin{split}
		\lim_{\epsilon\searrow0}\frac{\mathcal{W}^2_2\Big(\mathcal{L}^{(\epsilon)}_{\overline{t}_\epsilon},\mathcal{I}^{(\epsilon)}_{\overline{t}_\epsilon}\Big)}{\epsilon^2}&=2\sigma^2_0\left(g_0(\omega)-\sqrt{\frac{1}{2}\left(g^2_0(\omega)+\frac{h_0(\omega)}{2}\right)}\right)\\
		&=2\sigma^2_0\left(g_0(\omega)-\sqrt{g^2_0(\omega)-g_0(\omega)\frac{e^{-2\theta \omega}}{\alpha^2}}\right)\\
		&=2\sigma^2_0\left(\frac{\sigma^2}{2\theta\sigma_0^2}+ \frac{e^{-2\theta \omega}}{\alpha^2}-\sqrt{\frac{\sigma^4}{4\theta^2\sigma_0^4 }+\frac{\sigma^2e^{-2\theta \omega}}{2\theta\sigma_0^2\alpha^2}}\right)\,.
	\end{split}
	\]
	From  similar computations, one can see that if instead we consider  $t_\epsilon+r\omega_\epsilon,$ 
	where $w_\epsilon\to w>0$,  then
	$$H(r):=\lim_{\epsilon\searrow0}\frac{\mathcal{W}^2_2\Big(\mathcal{L}^{(\epsilon)}_{t_\epsilon+r\omega_\epsilon},\mathcal{I}^{(\epsilon)}_{t_\epsilon+r\omega_\epsilon}\Big)}{\epsilon^2}=2\sigma^2_0\left(\frac{\sigma^2}{2\theta\sigma_0^2}+ \frac{e^{-2\theta r\omega}}{\alpha^2}-\sqrt{\frac{\sigma^4}{4\theta^2\sigma_0^4 }+\frac{\sigma^2 e^{-2r\theta \omega}}{2\theta\sigma_0^2\alpha^2}}\right)\,.$$ 
	Next, we verify the limiting behaviour of  $H(r)$ when $r$ is close to $\infty$ and to $-\infty$ (see Figure 1 for the form of $P$ when $\theta=\sigma^2=\sigma_0^2=\omega=\alpha=1$). For the first case, we observe
	$$\lim_{r\to\infty}H(r)=2\sigma^2_0\left(\frac{\sigma^2}{2\theta\sigma_0^2}-\sqrt{\frac{\sigma^4}{4\theta^2\sigma_0^4 }}\right)=0\,.$$
	On the other hand, when  $r\to-\infty$, we first see
	\[
	\begin{split}
		H(r) &= 2\sigma^2_0\left(\frac{\sigma^2}{2\theta\sigma_0^2}+ \frac{e^{-2\theta r\omega}}{\alpha^2}-e^{-r\theta \omega}\sqrt{\frac{\sigma^4 e^{2r\theta \omega}}{4\theta^2\sigma_0^4 }+\frac{\sigma^2}{2\theta\sigma_0^2\alpha^2}}\right).
	\end{split}
	\]
	The latter implies that 
	$$\lim_{r\to-\infty}H(r)=2\sigma^2_0\lim_{r\to-\infty}\left(\frac{\sigma^2}{2\theta\sigma_0^2}+ \frac{e^{-2\theta r\omega}}{\alpha^2}-e^{-r\theta \omega}\sqrt{\frac{\sigma^2e^{2r\theta \omega}}{4\theta^2\sigma_0^4 }+\frac{\sigma^2}{2\theta\sigma_0^2 \alpha^2}}\right)=\infty\,.$$
	We conclude our proof by observing that $H(r)\ge 0$ is always positive  and taking square roots in both sides in the previous identity, that is $G_{\mathcal{W}_2}(r)=\sqrt{H(r)}$.

	Finally, we consider part (iii).  We first
	recall the   notation of part (ii), that is $a=\sigma_0^2, c_t=\sigma_0^2\Phi(t)$ and $b_t$ which is defined in \eqref{notpartii}. We  also recall that $\mathcal{L}^{(\epsilon)}_{t}$ and $\mathcal{I}^{(\epsilon)}_{t}$ have the same law as  $\mathcal{N}(\eta_t, \Xi_t)$ and $\mathcal{N}(\widetilde{\eta}_t, \widetilde{\Xi}_t)$,  respectively,   where $\eta_t=\widetilde{\eta}_t=(\mu_0, \mu_0\Phi(t))$ and the matrices $\Xi_t$ and $\widetilde{\Xi}_t$ are defined as in \eqref{matrices}.
	
	Hence, we use the following well-known properties  of the total variation distance (see for instance Lemma A.1 parts (ii) and (iv) in \cite{BJ1})  to obtain
	\[
	\begin{split}
		||\mathcal{L}^{(\epsilon)}_{t}-\mathcal{I}^{(\epsilon)}_{t}||_{TV}&=|| \mathcal{N}(\eta_t, \Xi_t)-\mathcal{N}(\widetilde{\eta}_t, \widetilde{\Xi}_t)||_{TV}\\
		&=|| \mathcal{N}(\eta_t-\widetilde{\eta}_t, \Xi_t)-\mathcal{N}(0, \widetilde{\Xi}_t)||_{TV}\\
		&=|| \mathcal{N}(0, \widetilde{\Xi}_t^{-1/2}\Xi_t\widetilde{\Xi}_t^{-1/2})-\mathcal{N}(0, \mathbf{I}_2)||_{TV}.
	\end{split}
	\]
	On the other hand, a similar computation as in \eqref{tomate}, allow us to deduce
	\[
	\widetilde{\Xi}_t^{-1/2}\Xi_t\widetilde{\Xi}_t^{-1/2}=\begin{pmatrix}
		1 & (ab_t)^{-1/2}c_t\\
		(ab_t)^{-1/2}c_t & 1
	\end{pmatrix},
	\]
	where $(ab_t)^{-1/2}c_t$ satisfies
	$$(ab_t)^{-1/2}c_t=\frac{\sigma_0\Phi(t)}{\sqrt{\epsilon^2\sigma^2\Phi^2(t)\int_0^{ t}\frac{{\rm d}u}{\Phi^2(u)}+\sigma_0^2\Phi^2(t)}} .$$ 
	Next for $\epsilon\in(0,1)$ and $\omega_\epsilon\to\omega\in \mathbb{R}$, we consider
	$t_\epsilon +r\omega_\epsilon$ 
	and observe 
	\[
	\begin{split}
		(ab_{t_\epsilon +r\omega_\epsilon})^{-1/2}c_{t_\epsilon +r\omega_\epsilon}&=\frac{\sigma_0\epsilon e^{-\theta r\omega_\epsilon+ \int_0^{t_\epsilon +r\omega_\epsilon}A(s){\rm d}s}}{\sqrt{\epsilon^2\sigma^2\Phi^2(t_\epsilon +r\omega_\epsilon)\int_0^{ t_\epsilon +r\omega_\epsilon}\frac{{\rm d}u}{\Phi^2(u)}+\sigma_0^2\epsilon^2 e^{-2\theta r\omega_\epsilon+ 2\int_0^{t_\epsilon +r\omega_\epsilon}A(s){\rm d}s}}}\\
		& \to\sigma_0\sqrt{\frac{2\theta}{\sigma^2 \alpha^2e^{2\theta rw}+2\theta\sigma_0^2 }}\,,\qquad\mbox{ as }\,\epsilon\to 0\,.
	\end{split}
	\]
	In other words,
	\[
	\lim_{\epsilon\to 0}||\mathcal{L}^{(\epsilon)}_{t_\epsilon +r\omega_\epsilon}-\mathcal{I}^{(\epsilon)}_{t_\epsilon +r\omega_\epsilon}||_{TV}=||\mathcal{N}(0, M_r)-\mathcal{N}(0,\mathbf{I}_2)||_{TV},
	\]
	as expected. This concludes  the proof.
\end{proof}

\subsection{The multidimensional case}

Recall that  $Q$ is a $d \times d$ Hurwitz-stable matrix and that the fundamental solution to \eqref{difeq} is given by
$\Phi(t)=e^{ t Q},$ for  $t\ge 0.$
We also recall  that the controllability rank condition \eqref{rankcond} is fulfilled and that  the process $X^{(\epsilon)}$ is the unique strong solution of  \eqref{SDE} with initial condition $X_0= \xi$,  a multivariate Gaussian distribution in $\mathbb{R}^{d}$ with mean  $\mu_0$ and covariance matrix $\Sigma_0$,  independent of the Brownian motion $B$. Hence the couple $(X_0,X_t)$ is  a  multivariate Gaussian distribution with covariance  matrix
$$
\Xi_t=
\begin{pmatrix}
	\Sigma_0 & \Sigma_0\, e^{ t Q^{T} } \\ 
	e^{ t Q } \,\Sigma_0  & \rho_t 
\end{pmatrix}
$$
where $\rho_t$ is defined in \eqref{covt}.  For a matrix $A$, we  define 
\[
f_t(A):= A \Sigma_0 A^{T} + \varsigma_t ,
\] 
where $\varsigma_t$ is defined in \eqref{vars}.
Thus the covariance matrix $\rho_t$ can be rewritten as follows
$$ \rho_t=e^{tQ}\Sigma_0e^{tQ^T}+\epsilon^2 \varsigma_t=\epsilon^2\left( \left(\epsilon^{-1}e^{tQ}\right)\Sigma_0\left(\epsilon^{-1}\left(e^{tQ}\right)^T\right)+\varsigma_t\right)=\epsilon^2 f_t \left(\epsilon^{-1}e^{tQ}\right). $$
We also recall that  $Y$ is an independent  copy of  $X_0$  and independent of the Brownian motion $B$, so the couple $(Y,X_t)$ is a multivariate Gaussian distribution with covariance  matrix given by
$$
\widetilde{\Xi}_t=
\begin{pmatrix}
	\Sigma_0 & 0 \\ 
	0  & \rho_t 
\end{pmatrix}.
$$
Before proceeding to the proof of Theorem \ref{theorem2}, we establish the following technical lemma that will be instrumental in the subsequent analysis.
\begin{lemma}\label{Jordan} Let $Q\in\mathbb{R}^{d\times d}$ be a Hurwitz-stable matrix.  Denote by $\vartheta$  the real part of any critical eigenvalue, and let $t_\epsilon$  and  $\omega_\epsilon$ be defined as in the statement of Theorem \ref{theorem2}. \begin{itemize}
		\item[i)] If there is a unique real critical eigenvalue, then the following limit holds
		\[
		\lim_{\epsilon \downarrow 0} \frac{e^{(t_{\epsilon}+r\omega_\epsilon) Q}}{\epsilon}=e^{-rw}\frac{|\vartheta|^{1-m}}{(m-1)!} P_1\mathbf{I}^{\ell}_mP_1^{-1}=e^{-rw}\Gamma\,,
		\]
		where $P_1$, $\mathbf{I}^{\ell}_m$ and $\Gamma$ are defined as in the statement of Theorem \ref{theorem2}.
		\item[ii)] In the complex case, the above  limit does not converge in general. However, there exist a subsequence $(\epsilon_n)_{n\ge 1}$ with $\epsilon_n\downarrow 0$ such that the limit  exists along this subsequence.
	\end{itemize}
\end{lemma}
\begin{proof} From the Jordan decomposition of $Q$, we have  $e^{tQ}=PJP^{-1}$, where
	$$J=\begin{pmatrix}
		e^{tB_1}& & \\ 
		&\ddots &  \\
		& & e^{tB_p} 
	\end{pmatrix},$$
	is the Jordan canonical form of $e^{tQ}$. Observe that all the others entries in $J$ are zero and each $e^{tB_j}$ is a square matrix of one of the following forms
	\begin{equation}\label{eq:Jordan}
		(i)\,\, e^{tB_j}=\begin{pmatrix}
			e^{\lambda_j t}&te^{\lambda_j t}&\dots&\frac{t^{m_j-1}}{(m_j-1)!}e^{\lambda_j t}\\ 
			&e^{\lambda_j t}&\dots&\frac{t^{m_j-2}}{(m_j-2)!}e^{\lambda_j t}\\
			& &\ddots & \vdots \\
			& & & e^{\lambda_j t } 
		\end{pmatrix}\,\mbox{ or } (ii)\,\, e^{tB_j}=\begin{pmatrix}e^{\Lambda_j t}&te^{\Lambda_j t}&\dots&\frac{t^{m_j-1}}{(m_j-1)!}e^{\Lambda_j t}\\ 
			&e^{\Lambda_j t}&\dots&\frac{t^{m_j-2}}{(m_j-2)!}e^{\Lambda_j t}\\
			& &\ddots & \vdots \\
			& & & e^{\Lambda_j t }\end{pmatrix}
	\end{equation}
	with 
	$$e^{\Lambda_j t}=e^{\alpha_jt} \begin{pmatrix}\cos\left(\beta_j t\right)&-\sin\left(\beta_j t\right)\\
		\sin\left(\beta_j t\right)& \cos\left(\beta_j t\right)\end{pmatrix}\,.$$
	This representation follows from the Jordan decomposition, where in (i) the eigenvalue $\lambda_j\in\mathbb{R}$ and $e^{tB_j}$ is a $m_j\times m_j$ matrix,  while in (ii) the eigenvalue $\lambda_j=\alpha_j+i\beta_j\in\mathbb{C}$ has a non zero imaginary part and $e^{tB_j}$ is a $2m_j\times 2m_j$ matrix. Since we are assuming that 
	\[
	\vartheta=\max \{ \Re(\lambda_i)\,:\,j=1,\dots,k \} <0\,,
	\] 
	it happens that $e^{tQ}\to 0$ as $t\to\infty$. In order to capture the block $e^{tB_j}$
	that decays at a slower rate, we note that there may exist multiple blocks associated with 
	$\vartheta$.  Let
	$$m:=\max\{m_j\,:\,\Re(\lambda_i)=\vartheta\}\geq 1\,, $$ 
	that is, $m$ is the largest among the $m_j$ associated with $\vartheta$, so that $t^{m-1}e^{\vartheta t}$  corresponds to the slowest decaying term appearing in $e^{tQ}$. Let  $t_\epsilon$  and  $\omega_\epsilon$ be as defined  in the statement of Theorem \ref{theorem2}.
	Then 
	$$\left(t_\epsilon + r w_\epsilon\right)^{m-1} e^{\vartheta(t_\epsilon + r w_\epsilon)}=\left(t_\epsilon + r w_\epsilon\right)^{m-1}\frac{\epsilon}{\left(\log\frac{1}{\epsilon}\right)^{m-1}}e^{-rw+\mathcal{O}(\epsilon)}\,,$$
	and since 
	$$\lim_{\epsilon\downarrow 0} \frac{\left(t_\epsilon + r w_\epsilon\right)^{m-1}}{\left(\log\frac{1}{\epsilon}\right)^{m-1}}=|\vartheta|^{1-m}\,,$$
	it follows that 
	$$\lim_{\epsilon\downarrow 0}\frac{\left(t_\epsilon + r w_\epsilon\right)^{m-1} e^{ \vartheta(t_\epsilon + r w_\epsilon)}}{\epsilon} =|\vartheta|^{1-m}e^{-rw}\,.$$
	If we have a block $e^{-tB_j}$ associated to $\vartheta$ and $m$ that has the real form (i), that is $\lambda_j=\vartheta$ and $m_j=m$, then
	$$\lim_{\epsilon\downarrow 0}\frac{ e^{ (t_\epsilon+rw_\epsilon) B_j}}{\epsilon}= e^{-rw}\frac{|\vartheta|^{1-m}}{(m-1)!} E^{(m)}_{1m}\,,$$
	where we recall that $E^{(m)}_{1m}$ is the $m\times m$ matrix with a 1 in the $(1,m)$-th entry and zeros elsewhere.
	
	If we have a block $e^{tB_j}$ associated to $\vartheta$ and $m$ that has the complex form (ii), that is $\alpha_j=\vartheta$ and $m_j=m$, then the respective rotation matrix (with $\sin\left(\beta_jt\right)$ and $\cos\left(\beta_j t\right)$) will stay bounded but will not converge, although one can see that we do have limits along subsequences. 
	
	Finally, all  blocks not associated with $\vartheta$ and $m$  converge to zero. This implies  that  the limit of $\epsilon^{-1}e^{ \left(t_\epsilon + r w_\epsilon\right) Q}$, as $\epsilon\downarrow 0$, exists  if and only if all  blocks associated with $\vartheta$ and $m$ are of the form (i),  assuming that there are $\ell\ge 1$ such blocks.  By appropriately rearranging, we can assume that all block associated with $\vartheta$ and $m$ appear first along the diagonal. Under this arrangement, the limit can be expressed as
	$$e^{-rw}\frac{|\vartheta|^{1-m}}{(m-1)!} P_1\mathbf{I}^{\ell}_mP_1^{-1}.$$   
	This completes the proof. \end{proof}

\begin{proof}[Proof of Theorem \ref{theorem2}] 
	
	We first prove part (i). Observe that
	$$ 
	\widetilde{\Xi}_t ^{-1}=  
	\begin{pmatrix} 
		\Sigma_0^{-1} & 0 \\ 
		0 & \rho_t^{-1} 
	\end{pmatrix}, $$
	since the block entries are covariance matrices themselves. Thus,  the matrix $\widetilde{\Xi}_t^{-1} \Xi_t$ has ones in the diagonal implying that  the term $\mathrm{tr}( \widetilde{\Xi}_t^{-1}\Xi_t )- 2d$ equals zero. Hence from the definition of the Kullback-Leibler divergence for Gaussian distributions, see \eqref{KLgaus}, we deduce
	\[
	D_{KL}\Big( \mathcal{L}^{(\epsilon)}_t  \Big| \mathcal{I}^{(\epsilon)}_t  \Big) = \frac{1}{2} \ln \left( \frac{\det  \widetilde{\Xi_t} }{\det \Xi_t } \right)=-\frac{1}{2}\ln\left( \frac{\det \Xi_t }{\det  \widetilde{\Xi_t} } \right)=-\frac{1}{2}\ln\left(\det ( \mathbf{I}_{2d}+ \epsilon \Lambda_t)\right)\,,
	\]
	where
	\[
	\begin{split}
		\epsilon \Lambda_t &= \left(\Xi_t - \widetilde{\Xi}_t\right)  \widetilde{\Xi}_t^{-1} = \begin{pmatrix} 0 & \Sigma_0\, e^{t Q^{T} }   \rho_t^{-1}  \\ e^{t Q }   & 0 \end{pmatrix}= \begin{pmatrix} 0 & \Sigma_0\, e^{t Q^{T} }  \epsilon^{-2}  f_t \left( \epsilon^{-1} e^{tQ}  \right)^{-1}  \\
			e^{tQ} &  0 \end{pmatrix}\, , 
	\end{split}
	\]
	and 
	$$ \mathbf{I}_{2d}+ \epsilon \Lambda_t= \begin{pmatrix}\mathbf{I}_{d} & \Sigma_0\, e^{t Q^{T} }  \epsilon^{-2}  f_t \left( \epsilon^{-1} e^{tQ}  \right)^{-1}  \\ e^{tQ}  & \mathbf{I}_{d} \end{pmatrix}\,.$$
	To deduce our goal, we use the identity 
	\begin{equation}\label{blockdet}
		\det\begin{pmatrix} \mathbf{I}_d & B_{12}\\ B_{21}& \mathbf{I}_d\end{pmatrix}=\det\left( \mathbf{I}_d-B_{21}B_{12}\right)=\det\left( \mathbf{I}_d-B_{12}B_{21}\right)
	\end{equation}
	for determinants of block matrices. In our setting,  
	$$B_{21}B_{12}=e^{tQ}\Sigma_0\, e^{t Q^{T} }  \epsilon^{-2}  f_t \left( \epsilon^{-1} e^{tQ}  \right)^{-1} = \left(\epsilon^{-1}e^{tQ}\right)\Sigma_0\, \left(\epsilon^{-1}e^{t Q^{T} }\right)  f_t \left( \epsilon^{-1} e^{tQ}  \right)^{-1}\,,$$
	which yields to 
	\begin{equation}\label{conv}
		\det\left(\mathbf{I}_{2d}+\epsilon \Lambda_t \right)=\det\left(\mathbf{I}_d- \left(\epsilon^{-1}e^{tQ}\right)\Sigma_0\, \left(\epsilon^{-1}e^{t Q^{T} }\right)  f_t \left( \epsilon^{-1} e^{tQ}  \right)^{-1}\right)\,.
	\end{equation}	
	
	Take $t=t_\epsilon+ rw_\epsilon$. In the complex case, that is when some leading eigenvalues of $Q$ are complex, we do not have convergence in the right-hand side of \eqref{conv}, however the convergence along a subsequence provided by Lemma \ref{Jordan} ii), and that $f_t(A)\to f(A):=A\Sigma_0A^T+\varsigma_{\infty}$, as $\epsilon\to 0$, allow us to conclude the window profile decorrelation. On the other hand, when the leading eigenvalue is real, by Lemma \ref{Jordan} i), the right-hand side of \eqref{conv} converges, as $\epsilon\to 0$,  to $$ \det\left( \mathbf{I}_d- e^{-2rw}\Gamma\Sigma_0\Gamma^Tf\left(e^{-rw}\Gamma\right)^{-1}\right)=\det\left( \mathbf{I}_d- e^{-2rw}\Gamma\Sigma_0\Gamma^T\left(e^{-2rw}\Gamma\Sigma_0\Gamma^T+ \varsigma_{\infty} \right)^{-1}\right)\,.$$
	
	Thus we obtain the following decorrelation profile for the Kullback-Leibler divergence,
	\[
	\begin{split}
		G_{KL}(r)&=-\frac{1}{2}\ln\left( \det\left( \mathbf{I}_d- e^{-2rw}\Gamma\Sigma_0\Gamma^T\left(e^{-2rw}\Gamma\Sigma_0\Gamma^T+ \varsigma_{\infty} \right)^{-1}\right)\right)\\
		&= -\frac{1}{2}\ln\left(\det\left( \mathbf{I}_d- \Gamma\Sigma_0\Gamma^T\left(\Gamma\Sigma_0\Gamma^T+ e^{2rw}\varsigma_{\infty} \right)^{-1}\right) \right)\,.
	\end{split}
	\]
	The first equality implies that  
	$$\lim_{r\to+\infty}G_{KL}(r)=-\frac{1}{2}\ln\left( \det\left( \mathbf{I}_d\right)\right)=0\,. $$
	
	To evaluate the limit as $r\to-\infty$, we use the second equality together with the Moore-Penrose pseudoinverse $M^+$ of a square matrix $M$, a concept that generalises the usual notion of the inverse of a square matrix, see for instance \cite{Mo,Pe}. The Moore-Penrose psudoinverse may be defined through limit
	\begin{equation}\label{moore-penrose}
		M^{+}:=\lim_{\delta\to 0}M^T\left(MM^T+\delta  \mathbf{I}_d\right)^{-1} =\lim_{\delta\to 0}\left(M^TM+\delta  \mathbf{I}_d\right)^{-1}M^T\,.
	\end{equation}
	For a self-contained introduction to the theory of Moore-Penrose pseudoinverses, we refer to \cite{BH}; see, in particular, Theorem 4.3 therein. Thus, if we denote $A=\Gamma \Sigma_0^{1/2}$ then  
	\[
	\begin{split}
		\Gamma\Sigma_0\Gamma^T\left(\Gamma\Sigma_0\Gamma^T+ e^{2rw}\varsigma_{\infty} \right)^{-1}&=AA^T\left(AA^T+ e^{2rw}\varsigma_{\infty} \right)^{-1} \\
		&= AA^T\left(\varsigma_\infty^{1/2}\left(\varsigma_\infty^{-1/2}AA^T\varsigma_\infty^{-1/2}+ e^{2rw}\mathbf{I}_d\right)\varsigma_\infty^{1/2} \right)^{-1}\\
		&= A\left[(\varsigma_\infty^{-1/2}A)^T\left((\varsigma_\infty^{-1/2}A)(\varsigma_\infty^{-1/2}A)^T+ e^{2rw}\mathbf{I}_d \right)^{-1}\right]\varsigma_\infty^{-1/2}\\
		&\to A\left(\varsigma_\infty^{-1/2}A\right)^+\varsigma_\infty^{-1/2}\,,\quad\mbox{ as $r\to-\infty$}\,.
	\end{split}
	\]
	Hence, $\det\left( \mathbf{I}_d- \Gamma\Sigma_0\Gamma^T\left(\Gamma\Sigma_0\Gamma^T+ e^{2rw} \varsigma_{\infty} \right)^{-1}\right)$ converges, as $r\to-\infty$,  to 
	\[
	\begin{split}
		\det\left( \mathbf{I}_d- A\left(\varsigma_\infty^{-1/2}A\right)^+\varsigma_\infty^{-1/2}\right)&=\det\left( \varsigma_\infty^{-1/2}\right)\det\left( \mathbf{I}_d- A\left(\varsigma_\infty^{-1/2}A\right)^+\varsigma_\infty^{-1/2}\right)\det\left( \varsigma_\infty^{1/2}\right)\\
		&= \det\left(\varsigma_\infty^{-1/2}\varsigma_\infty^{1/2}- \varsigma_\infty^{-1/2}A\left(\varsigma_\infty^{-1/2}A\right)^+\right)\\
		&=\det\left(\mathbf{I}_d- BB^+\right)\,,
	\end{split}
	\]
	where $B=\varsigma_\infty^{-1/2}A$. On one hand, $\mathbf{I}_d-BB^+$ is a projection (see Proposition 3.3 in \cite{BH}) and its determinant has to be $0$ or $1$. On the other hand, it is also a projection onto a proper subspace of $\mathbb{R}^d$, since ${\rm Ker}(\mathbf{I}_d-BB^+)={\rm Ran}(B)={\rm Ran}( \varsigma_\infty^{-1/2}\Gamma \Sigma_0^{1/2})$, which shows that it must be $0$. This concludes the proof that
	$$\lim_{r\to-\infty}G_{KL}(r)=+\infty\,.$$

	For the second part, and in analogy with the one-dimensional case,   we once again apply the mean-shift and scaling properties  of the total variation distance (see for instance Lemma A.1 parts (ii) and (iv) in \cite{BJ1})  to obtain
	\[
	\begin{split}
		||\mathcal{L}^{(\epsilon)}_{t}-\mathcal{I}^{(\epsilon)}_{t}||_{TV}&=|| \mathcal{N}(0, \widetilde{\Xi}_t^{-1/2}\Xi_t\widetilde{\Xi}_t^{-1/2})-\mathcal{N}(0, \mathbf{I}_{2d})||_{TV}.
	\end{split}
	\]
	We observe that the covariance matrix $\widetilde{\Xi}_t^{-1/2}\Xi_t\widetilde{\Xi}_t^{-1/2}$ satisfies
	$$\widetilde{\Xi}_t^{-1/2}\Xi_t\widetilde{\Xi}_t^{-1/2}=\begin{pmatrix}\mathbf{I}_{d} & \Sigma^{1/2}_0e^{ tQ^T}\rho_t^{-1/2} \\ \rho_t^{-1/2} e^{tQ}\Sigma_0^{1/2} & \mathbf{I}_{d} \end{pmatrix}\,.$$
	Moreover, we have
	$$ \rho_t^{-1/2} e^{tQ}=f_t\left(\epsilon^{-1}e^{tQ} \right)^{-1/2}\epsilon^{-1}e^{tQ} \,.$$
	Thus again from Lemma \ref{Jordan}, we obtain the window decorrelation in the complex case and the following decorrelation profile for the total variation distance in the real case:
	$$G_{TV}(r)= || \mathcal{N}(0, M(r))-\mathcal{N}(0, \mathbf{I}_{2d})||_{TV}\,,$$
	where 
	\[	 
	M(r)=\begin{pmatrix}
		\mathbf{I}_d &  \Sigma_0^{\frac{1}{2}} \Gamma^T  ( \Gamma \Sigma_0 \Gamma^T + e^{2r w} \varsigma_{\infty} )^{-\frac{1}{2}} 
		\\ ( \Gamma \Sigma_0 \Gamma^T + e^{2rw} \varsigma_{\infty} )^{-\frac{1}{2}} \Gamma \Sigma_0^{\frac{1}{2}} & \mathbf{I}_d  \end{pmatrix}\,.
	\]
	It is straightforward to check that 
	$$\lim_{r\to+\infty}G_{TV}(r)= || \mathcal{N}(0, \mathbf{I}_{2d})-\mathcal{N}(0, \mathbf{I}_{2d})||_{TV}=0\,.$$
	If we see $M(r)$ as a block matrix as in \eqref{blockdet}, then one can deduce that $B_{12}B_{21}$ equals
	\[
	\begin{split}
		\Sigma_0^{1/2}\Gamma^T\left(\Gamma\Sigma_0\Gamma^T+ e^{2rw}\varsigma_{\infty} \right)^{-1} \Gamma\Sigma_0^{1/2}&=A^T\left(AA^T+ e^{2rw}\varsigma_{\infty} \right)^{-1}A\\
		&=A^T\left(\varsigma_{\infty}^{1/2}\left(\varsigma_{\infty}^{-1/2}AA^T\varsigma_{\infty}^{-1/2}+ e^{2rw}\mathbf{I}_d\right)\varsigma_{\infty}^{1/2} \right)^{-1}A\\
		&=(\varsigma_{\infty}^{-1/2}A)^T\left(\varsigma_{\infty}^{-1/2}A(\varsigma_{\infty}^{-1/2}A)^T+ e^{2rw}\mathbf{I}_d\right)^{-1}\varsigma_{\infty}^{-1/2} A\\
		&=B^T\left(BB^T+e^{2rw}\mathbf{I}_d\right)^{-1}B\\
		&\to  B^+B\,,\quad\mbox{ as }r\to -\infty\,,
	\end{split}
	\]
	where once again the definition of the Moore-Penrose pseudoinverse \eqref{moore-penrose} is used. Hence, 
	$$\lim_{r\to-\infty}\det\left(\mathbf{I}_d-\Sigma_0^{1/2}\Gamma^T\left(\Gamma\Sigma_0\Gamma^T+ e^{2rw}\varsigma_{\infty} \right)^{-1} \Gamma\Sigma_0^{1/2} \right)= \det\left(\mathbf{I}_d-B^+B\right)=0\,,$$
	since $\mathbf{I}_d-B^+B$ is also a projection onto a proper subspace of $\mathbb{R}^d$. This implies that the support of $\mathcal{N}(0, M(r))$ is concentrating on a proper subspace of $\mathbb{R}^d$, which implies that it is getting singular with respect to the Lebesgue measure. Therefore, 
	$$\lim_{r\to-\infty}G_{TV}(r)= \lim_{r\to-\infty}|| \mathcal{N}(0, M(r))-\mathcal{N}(0, \mathbf{I}_{2d})||_{TV}=1\,.$$
	This completes the proof.
\end{proof}

\section*{Acknowledgments}
\noindent
This research was supported by several organizations. The research of JCP was supported by SECIHTI through the {\it Proyecto Ciencia de Frontera CF-2023-I-2566}. LPRP's work was funded by {\it DGAPA-UNAM PREI 2024} project. SIL received funding by {\it DGAPA-UNAM PAPIIT IN114425} project. JCP and SIL gratefully acknowledge the support and facilities provided by IM-UFRJ during the development of this work, which was also supported by {\it CNPq} through grant {\it Universal 403423/2023-6}. Similarly, SIL and LPRP  express their gratitude to CIMAT for the hospitality and resources made available throughout the course of this research. Finally, LPRP thanks UNAM by the support and amenities provided.


\begin{thebibliography}{99}
	
		
\bibitem{AlDi} {\sc Aldous, D., Diaconis, P.} Shuﬄing cards and stopping times. {\it Amer. Math. Monthly} {\bf 93} no. 5, 333--348 (1986).	

\bibitem{Bao} {\sc Bao, Z., Cipolloni, G., Erd\H{o}s, L., Henheik, J., Kolupaiev, O.} Decorrelation transition in the Wigner minor process.
{\it Probab. Theory Related Fields}, (2025).
	
\bibitem{BH} {\sc Barata, J.C.A., Hussein, M.S.} The Moore-Penrose Pseudoinverse: A Tutorial Review of the Theory. {\it Braz. J. Phys.} {\bf 42}, 146--165, (2012).

\bibitem{BB}{\sc Baraud, Y., Birg\'e, L.} Rho-estimators revisited: General theory and applications. {\it  Ann. Statist.} { \bf 46}, no. 6B, 3767--3804, (2018).

\bibitem{BBS}{\sc Baraud, Y., Birg\'e, L., Sart, M.} A new method for estimation and model selection: rho-estimation. {\it Inventiones mathematicae,} {\bf 207}, 425--517, (2017) .

\bibitem{BJ} {\sc Barrera, G.,  Jara, M.} Abrupt convergence of stochastic small perturbations of one dimensional dynamical systems. {\it J. Stat. Phys.} {\bf 163},  no. 1, 113--138, (2016).

\bibitem{BJ1} {\sc Barrera, G.,  Jara, M.} Thermalisation for small random perturbation of hyperbolic dynamical systems. {\it Ann. Appl. Probab.}  {\bf 30},  no. 3, 1164--1208, (2020).

\bibitem{BLPP1} {\sc Barrera, G.; L\'opez, S.I.; Pardo, J.C.; Pimentel, L.} Abrupt decorrelation for non-linear Langevin systems with small random perturbations. {\it Work in progress},  (2026).

\bibitem{BY} {\sc Barrera, J., Ycart, B.} Bounds for left and right window cutoffs. {\it  ALEA Lat. Am. J. Probab. Math. Stat.} {\bf 11}, no. 2, 445--458, (2014).

\bibitem{DPZAB}{\sc Da Prato, G., Zabczyk.} {\it Stochastic Equations in Infinite Dimensions.} Cambridge University Press, (1992).

\bibitem{DL} {\sc  Dowson, D.C., Landau, B.V.} The Fr\'echet distance between multivariate normal distributions. {\it Journal of Multivariate Analysis} {\bf 12}, no. 3, 450--455, (1982).

\bibitem{Ferr} {\sc Ferrari, P.L., Frings, R.} Perturbed GUE Minor Process and Warren's Process with Drifts. {\it J. Stat. Phys.} {\bf 154}, 356--378, (2014).

\bibitem{Forr}{\sc Forrester, P.J., Nagao, T.} Determinantal Correlations for Classical Projection Processes. {\it J. Stat. Mech.} P08011, (2011).

\bibitem{KS91} {\sc Karatzas, I., Shreve, S.} {\it Brownian motion and stochastic calculus.}  Springer-Verlag, New-York, (1991).


\bibitem{Lec}{\sc Lecestre, A.} Robust estimation for ergodic Markovian processes. {Preprint}, available at {\tt https://arxiv.org/abs/2307.03666}, (2023). 

\bibitem{Mo} {\sc Moore, E. H.} On the reciprocal of the general algebraic matrix. {\it Bulletin of the American Mathematical Society} {\bf 26}, 394--395, (1920).

\bibitem{Par}{\sc Pardo, L.} {\it Statistical Inference Based on Divergence Measures}, Chapman and Hall/CRC, New York, (2006).

\bibitem{Pav}{\sc Pavliotis, G. A.} {\it Stochastic processes and applications}, Texts in Applied Mathematics, 60, Springer, New York, (2014) .

\bibitem{Pe} {\sc Penrose, R.} A generalized inverse for matrices, {\it Proceedings of the Cambridge Philosophical Society} {\bf 51}, 406--413, (1955).

\bibitem{PolWu}{\sc Polyanskiy, Y., Wu, Y.} Dissipation of information in channels with input constraints. {\it IEEE Transcations on Information Theory}  {\bf 62} (1), 35--55, (2016).

\end{thebibliography}
\end{document}